\documentclass[letterpaper,10pt,conference]{IEEEconf}
\IEEEoverridecommandlockouts                              
\usepackage{amsmath,amssymb,amsfonts}
\allowdisplaybreaks
\usepackage{bm}
\usepackage{algorithm}
\usepackage{algpseudocode}
\usepackage{graphicx}
\usepackage{subcaption}
\usepackage{xcolor}
\makeatletter
\let\NAT@parse\undefined
\makeatother
\usepackage[hidelinks]{hyperref}

\usepackage{ntheorem} 
\newtheorem{theoremii}{Theorem}[section]

\newtheorem{assumptionii}{Assumption}[section]

\newtheorem{lemmaiii}{Lemma}[theoremii]

\newtheorem{remark}{Remark}
\newtheorem{definitionii}{Definition}[section]

\def\BibTeX{{\rm B\kern-.05em{\sc i\kern-.025em b}\kern-.08em
    T\kern-.1667em\lower.7ex\hbox{E}\kern-.125emX}}

\begin{document}

\title{\LARGE \bf Second-order Properties of Noisy Distributed Gradient Descent\\
}

\author{Lei Qin, Michael Cantoni, and Ye Pu
\thanks{This work was supported by a Melbourne Research Scholarship and the Australian Research Council (DE220101527 and DP210103272).}
\thanks{L. Qin, M. Cantoni, and Y. Pu are with the Department of Electrical and Electronic Engineering, University of Melbourne, Parkville VIC 3010, Australia \texttt{\small leqin@student.unimelb.edu.au, \{cantoni, ye.pu\}@unimelb.edu.au}.
}}

\maketitle

\begin{abstract}
We study a fixed step-size noisy distributed gradient descent algorithm for solving optimization problems in which the objective is a finite sum of smooth but possibly non-convex functions. Random perturbations are introduced to the gradient descent directions at each step to actively evade saddle points. Under certain regularity conditions, and with a suitable step-size, it is established that each agent converges to a neighborhood of a local minimizer and the size of the neighborhood depends on the step-size and the confidence parameter. A numerical example is presented to illustrate the effectiveness of the random perturbations in terms of escaping saddle points in fewer iterations than without the perturbations.


\end{abstract}


\section{Introduction}
\label{sec: Introduction}
We consider the optimization problem
\begin{align}
    \label{eq: Unconstrained optimization problem}
    &\min_{\mathbf{x} \in \mathbb{R}^{n}} f(\mathbf{x}) \triangleq \min_{\mathbf{x} \in \mathbb{R}^{n}} \sum_{i=1}^{m} f_{i}(\mathbf{x}),
\end{align}
where each $f_{i}: \mathbb{R}^{n} \rightarrow \mathbb{R}$ is \textit{smooth} but possibly \emph{non-convex}, and $\mathbf{x} \in \mathbb{R}^{n}$ is the decision vector. The aim is to employ $m$ agents to iteratively solve the optimization problem in \eqref{eq: Unconstrained optimization problem}, over an undirected and connected network graph $\mathcal{G}(\mathcal{V},\mathcal{E})$. Each agent $i \in \mathcal{V}:=\{1,\ldots,m\}$ only knows the corresponding function $f_{i}$ and its gradient. The pair of agents $(i,j) \in \mathcal{V} \times \mathcal{V}$ is able to directly exchange information if and only if $(i,j)\in\mathcal{E}$. Collaborative distributed optimization over a network is of significant interest in the contexts of control, learning and estimation, particularly in large-scale system scenarios, such as unmanned vehicle systems \cite{dong2018theory}, electric power systems \cite{qiu2009literature}, transit systems \cite{gao2019reinforcement}, and wireless sensor networks \cite{nedic2018distributed}.

In the field of optimization, two primary classes of distributed methods can be identified: dual decomposition methods and consensus-based methods. Dual decomposition methods involve minimizing an augmented Lagrangian formulated on the basis of constraints that enforce agreement between agents, via iterative updates of the corresponding primal and dual variables \cite{boyd2011distributed}. The distributed dual decomposition algorithm in \cite{terelius2011decentralized} involves agents alternating between updating their primal and dual variables and communicating with their neighbors. In \cite{shi2014linear}, it is established that distributed alternating direction method of multipliers (ADMM) exhibits linear convergence rates in \emph{strongly convex} settings. Consensus-based methods can be traced back to the distributed computation models proposed in \cite{tsitsiklis1986distributed}, which seek to eliminate agent disagreements through local iterate exchange and weighted averaging to achieve consensus. This idea underlies the distributed (sub)gradient methods proposed in \cite{nedic2009distributed} and \cite{nedic2010constrained} to solve problem \eqref{eq: Unconstrained optimization problem} with all $f_i$ convex. In the case of diminishing step-size, each agent converges to an optimizer \cite{nedic2010constrained}; with constant step-size, convergence is typically faster, but only to the vicinity of an optimizer \cite{nedic2009distributed}.

The focus in this paper is on consensus-based distributed (sub)gradient methods, for the simplicity as first-order methods, enabling easy adaptation to diverse situations. In particular, a fixed step-size \textbf{D}istributed \textbf{G}radient \textbf{D}escent (\textbf{DGD}) algorithm is considered, as an instance of the distributed (sub)gradient methods. In unperturbed form, the update for each agent $i\in \mathcal{V}$ at iteration $k$ is given by
\begin{align}
\label{eq: DGD}
    \hat{\mathbf{x}}^{k+1}_{i} = \sum_{j=1}^{m} \mathbf{W}_{ij}\hat{\mathbf{x}}^{k}_j - \alpha\nabla f_{i}(\hat{\mathbf{x}}^{k}_{i}),
\end{align}
where $\alpha > 0$ is the constant step-size, $\nabla f_{i}$ is the gradient of $f_{i}$, $\hat{\mathbf{x}}_{i} \in \mathbb{R}^{n}$ is the local copy of the decision vector $\mathbf{x}$ at agent $i \in \mathcal{V}$, and $\mathbf{W}_{ij}$ is the scalar entry in the $i$-th row and $j$-th column of a given mixing matrix $\mathbf{W} \in \mathbb{R}^{m\times m}$. The mixing matrix is consistent with the graph $\mathcal{G}(\mathcal{V},\mathcal{E})$, in the sense that $\mathbf{W}_{ii} > 0$ for all $i \in \mathcal{V}$, $\mathbf{W}_{ij} > 0$ if $(i,j) \in \mathcal{E}$, and $\mathbf{W}_{ij} = 0$ otherwise. The convergence rates of fixed step-size and diminishing step-size \textbf{DGD} algorithms in \emph{strongly convex} settings are examined in \cite{yuan2016convergence} and \cite{tsianos2012distributed}, respectively.
Several variants of \textbf{DGD} have been proposed in \emph{(strongly) convex} settings. Nesterov momentum is used in the distributed gradient descent update with diminishing step-sizes in \cite{jakovetic2014fast} to improve convergence rates. Inexact proximal-gradient method is considered in \cite{chen2012fast} for problems involving non-smooth functions. To reach exact consensus with a constant step-size, 
the gradient tracking method \cite{shi2015extra,nedic2017achieving} and \cite{nedic2017geometrically} is used to neutralize $-\alpha \nabla f_{i}$ in \eqref{eq: DGD}, as the gradient descent part cannot vanish itself.

In \emph{non-convex} settings, gradient descent methods may face issues due to saddle points, as converging to a first-order stationary point does not guarantee the local minimality. While Hessian-based methods can avoid saddle points, their computational cost can be prohibitive for large-scale problems. In \cite{zeng2018nonconvex}, the fixed step-size \textbf{DGD} algorithm is shown to retain the property of convergence to a neighborhood of a consensus stationary solution under some regularity assumptions in \emph{non-convex} settings. It is also shown in \cite{daneshmand2020second} that \textbf{DGD} with a constant step-size converges almost surely to a neighborhood of second-order stationary solutions. However, this requires random initialization to avoid the zero measure manifold of saddle point attraction, and moreover, the underlying analysis does not support techniques for actively escaping saddle points.

Recently, it has been shown that standard (centralized) gradient descent methods can take exponential time to escape saddle points \cite{du2017gradient}, while noise (random perturbations) have been proven effective for escaping saddle points in \emph{non-convex} optimization. The Noisy Gradient Descent algorithm in \cite{ge2015escaping} and \cite{jin2017escape} is proven to be able to escape saddle points efficiently while converging to a neighborhood of a local minimizer with high probability. However, all of these works are limited to centralized methods. In \cite{swenson2022distributed} it is shown that distributed stochastic gradient descent converges to local minima almost surely when diminishing step sizes are used. Consider constant step-size, the diffusion strategy with stochastic gradient in \cite{vlaski2021distributed1} and \cite{vlaski2021distributed2} only returns approximately second-order stationary points rather than an outcome that lies in a neighborhood of a local minimizer with controllable size.

In this paper, the main contribution is that we analyze a fixed step-size noisy distributed gradient descent (\textbf{NDGD}) algorithm for solving the optimization problem in \eqref{eq: Unconstrained optimization problem} by expanding upon and combining ideas from \cite{ge2015escaping} and \cite{jin2017escape} on centralized stochastic gradient descent, and from \cite{yuan2016convergence,zeng2018nonconvex,daneshmand2020second} on distributed unperturbed gradient descent. In this combination, random perturbations are added to the gradient descent directions at each step to actively evade saddle points. It is established that under certain regularity conditions, and with a suitable step-size, each agent converges to a neighborhood of a local minimizer. In particular, we determine a probabilistic upper bound for the distance between the iterate at each agent and the set of local minimizers after a sufficient number of iterations. A numerical example is presented to illustrate the effectiveness of the algorithm in terms of escaping from the vicinity of a saddle point in fewer iterations than the standard (i.e., unperturbed) fixed step-size \textbf{DGD} method.
 
\subsection{Notation}
\label{sec: Notation}
Let $\mathbf{I}_n$ denote the $n\times n$ identity matrix, $\bm{1}_n$ denote the $n$-vector with all entries equal to $1$, and $\mathbf{A}_{ij}$ denote the entry in the $i$-th row and $j$-th column of the matrix $\mathbf{A}$. For a square symmetric matrix $\mathbf{B}$, we use $\lambda_{\mathrm{min}}(\mathbf{B})$, $\lambda_{\mathrm{max}}(\mathbf{B})$ and $\|\mathbf{B}\|$ to denote its minimum eigenvalue, maximum eigenvalue and spectral norm, respectively. The Kronecker product is denoted by $\otimes$. The distance from the point $\mathbf{x} \in \mathbb{R}^{n}$ to a given set $\mathcal{Y} \subseteq \mathbb{R}^{n}$ is denoted by $\textup{dist}(\mathbf{x}, \mathcal{Y}):= \inf_{\mathbf{y} \in \mathcal{Y}}\|\mathbf{x} - \mathbf{y}\|$. We say that a point $\mathbf{x}$ is $\delta$-close to a point $\mathbf{y}$ (resp., a set $\mathcal{Y}$) if $\textup{dist}(\mathbf{x},\mathbf{y}) \le \delta$ (resp., $\textup{dist}(\mathbf{x},\mathcal{Y}) \le \delta$). We use the Bachmann–Landau (asymptotic) notations including $\mathcal{O}(g(x,y))$, $\Omega(g(x,y))$ and $\Theta(g(x,y))$ to hide dependence on variables other than $x$ and $y$.


\section{Problem Setup and Supporting Results}
\label{sec: Problem Setup and Preliminaries}
In this section, we present a reformulation of the optimization problem defined in \eqref{eq: Unconstrained optimization problem} and provide a list of assumptions used in subsequent analysis. Then, we briefly recall aspects of the fixed step-size \textbf{DGD} algorithm \eqref{eq: DGD} and present some existing results for \emph{non-convex} optimization problems. Next, some intermediate results are derived to establish certain properties of the local minimizers of $f$ defined in \eqref{eq: Unconstrained optimization problem} (see Theorem \ref{the: Local minimizers for Q}), on the basis of a collection of supporting lemmas (see Lemma~\ref{lem: consensual bound on xQ*}-\ref{lem: arbitrarily close}).

\subsection{Problem Setup}
\begin{definitionii}
\label{def: First order stationary point}
    For differentiable function $h$, a point $\mathbf{x}$ is said to be first-order stationary if $\|\nabla h(\mathbf{x})\| = 0$.
\end{definitionii}

\begin{definitionii}
\label{def: Local minimizer and saddle}
    For twice differentiable function $h$, a first-order stationary point $\mathbf{x}$ is: (i) a local minimizer, if $\nabla^{2} h(\mathbf{x}) \succ 0$; (ii) a local maximizer, if $\nabla^{2} h(\mathbf{x}) \prec 0$; and (iii) a saddle point if $\lambda_{\min}(\nabla^{2} h(\mathbf{x})) < 0$ and $\lambda_{\max}(\nabla^{2} h(\mathbf{x})) > 0$.
\end{definitionii}

By introducing additional local variables, the optimization problem in \eqref{eq: Unconstrained optimization problem} can be reformulated as

\begin{gather}
    \begin{gathered}
        \label{eq: Constrained optimization problem}
        \min_{\hat{\mathbf{x}} \in (\mathbb{R}^{n})^{m}} F(\hat{\mathbf{x}}) \triangleq \min_{\hat{\mathbf{x}} \in (\mathbb{R}^{n})^{m}} \sum_{i=1}^m f_{i}(\hat{\mathbf{x}}_{i}),\\
        \textup{s.t. } \hat{\mathbf{x}}_{i} = \hat{\mathbf{x}}_{j} \text{ for all }(i,j)\in \mathcal{E},
    \end{gathered}    
\end{gather}
where $\hat{\mathbf{x}}_{i} \in \mathbb{R}^{n}$ is the local copy of the decision vector $\mathbf{x}$ at agent $i \in \mathcal{V}$, and $\hat{\mathbf{x}} = [\hat{\mathbf{x}}_{1}^{T},\cdots, \hat{\mathbf{x}}_{m}^{T}]^{T} \in (\mathbb{R}^{n})^{m}$.

\begin{assumptionii}[Local regularity]
\label{ass: Strict saddle}
    The function $f$ in \eqref{eq: Unconstrained optimization problem} is such that for all first-order stationary points $\mathbf{x}$, either $\lambda_{\min} (\nabla^{2} f (\mathbf{x})) > 0$ (i.e., $\mathbf{x}$ is a local minimizer), or $\lambda_{\min} (\nabla^{2} f (\mathbf{x})) < 0$ (i.e., $\mathbf{x}$ is a saddle point or a maximizer).
\end{assumptionii}

\begin{assumptionii}[Lipschitz gradient]
\label{ass: Smoothness gradient lipschitz} Each objective $f_{i}$ has $L_{f_i}^{g}$-Lipschitz continuous gradient, i.e.,
\begin{align*}
    \|\nabla f_{i}(\mathbf{x})
    - \nabla f_{i}(\mathbf{y})\| \le L_{f_i}^{g}
    \|\mathbf{x}-\mathbf{y}\|
\end{align*}
for all $\mathbf{x},\mathbf{y}\in\mathbb{R}^{n}$ and each $i \in \mathcal{V}$.
\end{assumptionii}

\begin{assumptionii}[Lipschitz Hessian]
\label{ass: Smoothness hessian lipschitz} 
Each objective $f_{i}$ has $L_{f_i}^{H}$-Lipschitz continuous Hessian, i.e.,
\begin{align*}
    \|\nabla^{2} f_{i}(\mathbf{x})
    - \nabla^{2} f_{i}(\mathbf{y})\| \le L_{f_i}^{H}
    \|\mathbf{x}-\mathbf{y}\|
\end{align*}
for all $\mathbf{x},\mathbf{y}\in\mathbb{R}^{n}$ and each $i \in \mathcal{V}$.
\end{assumptionii}

If Assumption \ref{ass: Smoothness gradient lipschitz} holds, then $F$ defined in \eqref{eq: Constrained optimization problem} has $L_{F}^{g}$-Lipschitz continuous gradient 
with $L_{F}^{g} = \max_{i} \{L_{f_{i}}^{g}\}$. Further, if Assumption \ref{ass: Smoothness hessian lipschitz} holds, then $F$ has $L_{F}^{H}$-Lipschitz continuous Hessian 
with $L_{F}^{H} = \max_{i} \{L_{f_{i}}^{H}\}$.

\begin{assumptionii}[Coercivity and properness]
\label{ass: Coercivity}
    Each local objective $f_i$ is coercive (i.e., its sublevel set is compact) and proper (i.e., not everywhere infinite).
\end{assumptionii}


\subsection{Distributed Gradient Descent}
\begin{assumptionii}[Network]
\label{ass: Network}
    The undirected graph $\mathcal{G}(\mathcal{V},\mathcal{E})$ is connected.
\end{assumptionii}

The \textbf{DGD} algorithm in \eqref{eq: DGD}, with constant step-size $\alpha>0$, can be written in a matrix/vector form as
\begin{align}
\label{eq: DGD matrix form}
    \hat{\mathbf{x}}^{k+1} = \hat{\mathbf{W}} \hat{\mathbf{x}}^{k} - \alpha \nabla F(\hat{\mathbf{x}}^{k}),
\end{align}
where $\hat{\mathbf{W}} := \mathbf{W} \otimes \mathbf{I}_n$. Note that from this point on, the mixing matrix $\mathbf{W}$ is taken to be symmetric, doubly stochastic and strictly diagonally dominant, i.e., $\mathbf{W}_{ii} > \sum_{j \ne i} \mathbf{W}_{ij}$ for all $i \in \mathcal{V}$. Thus, $\mathbf{W}$ is positive definite by the Gershgorin circle theorem. As proposed in some early works, including \cite{yuan2016convergence, zeng2018nonconvex, daneshmand2020second}, we can analyze the convergence properties using an auxiliary function. Let the $Q_{\alpha}$ denote the auxiliary function,
\begin{align}
\label{eq: Q}
    Q_{\alpha}(\hat{\mathbf{x}}) &= \sum_{i=1}^m f_{i}(\hat{\mathbf{x}}_{i}) + \frac{1}{2\alpha}\sum_{i=1}^m\sum_{j=1}^m(\mathbf{I}_{m} - \mathbf{W})_{ij} (\hat{\mathbf{x}}_{i})^{T}(\hat{\mathbf{x}}_{j}) \nonumber\\
    &= F(\hat{\mathbf{x}}) + \frac{1}{2\alpha}\|\hat{\mathbf{x}}\|^{2}_{\mathbf{I}_{mn}-\hat{\mathbf{W}}},
\end{align}
consisting of the objective function in \eqref{eq: Constrained optimization problem} and a quadratic penalty, which depends on the step-size and the mixing matrix. We use $\hat{\mathbf{x}}^{*}$ to denote a local minimizer of $Q_{\alpha}$. Note that the \textbf{DGD} update \eqref{eq: DGD matrix form} applied to \eqref{eq: Constrained optimization problem} can be interpreted as an instance of the standard gradient descent algorithm applied to \eqref{eq: Q}, i.e.,
\begin{align}
\label{eq: GD}
    \hat{\mathbf{x}}^{k+1} = \hat{\mathbf{x}}^{k} - \alpha \nabla Q_{\alpha}(\hat{\mathbf{x}}^{k}).
\end{align}
Thus, iteratively running \eqref{eq: DGD matrix form} and \eqref{eq: GD} from the same initialization yields the same sequence of iterates.

If Assumption \ref{ass: Smoothness gradient lipschitz} holds, then $Q_{\alpha}$ defined in \eqref{eq: Q} has $L_{Q_{\alpha}}^{g}$-Lipschitz continuous gradient 
with $L_{Q_{\alpha}}^g = L_{F}^{g} + \alpha^{-1} (1-\lambda_{\min}(\mathbf{W})) = \max_{i}\{L_{f_{i}}^{g}\} + \alpha^{-1} (1-\lambda_{\min}(\mathbf{W}))$. We have that $1-\lambda_{\min}(\mathbf{W}) \ge 0$ because the spectrum of a symmetric, positive definite and doubly stochastic matrix is contained in the interval $(0,1]$ by the Perron–Frobenius theorem, with $1$ being the only largest eigenvalue (the Perron root). Further, if Assumption \ref{ass: Smoothness hessian lipschitz} holds, then $Q_{\alpha}$ has $L_{Q_{\alpha}}^{H}$-Lipschitz continuous Hessian
with $L_{Q_{\alpha}}^{H} = \max_{i}\{L_{f_{i}}^{H}\} =  L_{F}^{H}$.

\subsection{Relationships between Local minimizers of $f$ and \texorpdfstring{$Q_{\alpha}$}{Lg}}
In this section, we prove that $\hat{\mathbf{x}}_{i}^{*}$, the component of a local minimizer $\hat{\mathbf{x}}^{*}$ of $Q_{\alpha}$ associated with agent $i$, can be made arbitrarily close to the set of local minimizers $\mathbf{x}^{*}$ of $f$ by choosing sufficiently small $\alpha > 0$. The proof is based on a collection of intermediate results. First, Lemma~\ref{lem: consensual bound on xQ*} shows that by choosing sufficiently small step-size $\alpha > 0$, at a local minimizer of $Q_{\alpha}$, the component corresponding to agent $i$ can be arbitrarily close to $\Bar{\mathbf{x}}^{*}$, where $\Bar{\mathbf{x}}^{*}$ denotes the average of $\hat{\mathbf{x}}_{i}^{*}$ across all agents. Lemma~\ref{lem: first order guarantee for xhat*} and Lemma~\ref{lem: second order guarantee for xhat*} show that given $\epsilon_{g}>0$ and $\epsilon_{H}>0$, one can always find constant step-size $\alpha > 0$ such that $\| \nabla f(\Bar{\mathbf{x}}^{*}) \| \le \epsilon_{g}$ and $\lambda_{\min}( \nabla^{2} f(\Bar{\mathbf{x}}^{*})) \ge -\epsilon_{H}$. Finally, Lemma~\ref{lem: arbitrarily close} shows that for each agent, $\hat{\mathbf{x}}_{i}^{*}$ can be made arbitrarily close to a local minimizer of $f$, by choosing a sufficiently small step-size $\alpha > 0$. Let $\mathcal{X}_{f}^{*}$ and $\hat{\mathcal{X}}_{Q_{\alpha}}^{*}$ denote the set of local minimizers of $f$ and $Q_{\alpha}$, respectively:
\begin{align}
\label{eq: minimizer set}
    \begin{aligned}
        \mathcal{X}_{f}^{*} :=\{\mathbf{x} \in \mathbb{R}^{n}:&~\nabla f(\mathbf{x})=0,~\nabla^2 f(\mathbf{x}) \succ 0\},\\
        \hat{\mathcal{X}}_{Q_{\alpha}}^{*} := \{\hat{\mathbf{x}} \in (\mathbb{R}^{n})^{m}:&~\nabla Q_{\alpha}(\hat{\mathbf{x}}) = 0,~\nabla^2 Q_{\alpha}(\hat{\mathbf{x}}) \succ 0\}. 
    \end{aligned}
\end{align}

\begin{lemmaiii}
\label{lem: consensual bound on xQ*}
    Let Assumption~\ref{ass: Network} hold. Given $\alpha > 0$, let $\hat{\mathbf{x}}^{*}$ be a local minimizer of $Q_{\alpha}$. Then, for each $i\in\mathcal{V}$,
    \begin{align*}
        \|\hat{\mathbf{x}}_{i}^{*} - \bar{\mathbf{x}}^{*}\| \le \alpha \cdot \frac{\|\nabla F(\hat{\mathbf{x}}^{*}) \|}{1-\lambda_{2}},
    \end{align*}
    where $\bar{\mathbf{x}}^{*} =  \frac{1}{m} (\mathbf{1}_{m} \otimes \mathbf{I}_{n})^{T} \hat{\mathbf{x}}^{*}$, and $0 < \lambda_{2} < 1$ is the second-largest eigenvalue value of $\mathbf{W}$.
\end{lemmaiii}

\begin{proof}
    Since $\hat{\mathbf{x}}^{*}$ is a local minimizer of $Q_{\alpha}$, we have $\nabla Q_{\alpha}(\hat{\mathbf{x}}^{*}) = \nabla F(\hat{\mathbf{x}}^{*}) + \alpha^{-1}(\mathbf{I}_{mn} - \hat{\mathbf{W}})\hat{\mathbf{x}}^{*}=0$, and thus,
    \begin{align*}
        \hat{\mathbf{x}}^{*} = \hat{\mathbf{W}} \hat{\mathbf{x}}^{*} - \alpha \nabla F(\hat{\mathbf{x}}^{*}) = \hat{\mathbf{W}}^{s+1} \hat{\mathbf{x}}^{*} - \alpha\sum_{t=0}^{s} \hat{\mathbf{W}}^{t} \nabla F(\hat{\mathbf{x}}^{*})
    \end{align*}
    for all $s \in \mathbb{N}$. Therefore, we have that 
    $\|\hat{\mathbf{x}}^{*} - \hat{\mathbf{W}}^{s+1} \hat{\mathbf{x}}^{*}\| = \alpha \|\sum_{t=0}^{s} \hat{\mathbf{W}}^{t} \nabla F(\hat{\mathbf{x}}^{*})\|
        \le \alpha \sum_{t=0}^{s} \|\hat{\mathbf{W}}^{t} \nabla F(\hat{\mathbf{x}}^{*})\|$.
    Now, since $(\bm{1}_{m} \otimes \mathbf{I}_{n})^{T}  (\mathbf{I}_{mn} -\hat{\mathbf{W}})\hat{\mathbf{x}}^{*} = 0$, we have $(\bm{1}_{m} \otimes \mathbf{I}_{n})^{T} \nabla F(\hat{\mathbf{x}}^{*}) = (\bm{1}_{m} \otimes \mathbf{I}_{n})^{T} (\nabla Q_{\alpha}(\hat{\mathbf{x}}^{*}) - \alpha^{-1} (\mathbf{I}_{mn} -\hat{\mathbf{W}})\hat{\mathbf{x}}^{*}) = 0$. Further, by connectivity of the underlying communication graph (see Assumption~\ref{ass: Network}) and the Perron–Frobenius theorem, $0 < \lambda_2 < 1$, and as such,
    \begin{align*}
        \|\hat{\mathbf{x}}_{i}^{*} - \hat{\mathbf{W}}^{s+1} \hat{\mathbf{x}}^{*}\| &\le \alpha \sum_{t=0}^{s} \|(\mathbf{W}^{t} - \frac{\bm{1}_{m}\bm{1}_{m}^{T}}{m}) \otimes \mathbf{I}_{n} \cdot \nabla F(\hat{\mathbf{x}}^{*})\|\\
        &\le \alpha \sum_{t=0}^{s} \lambda_{2}^{t} \|\nabla F(\hat{\mathbf{x}}^{*})\|,
    \end{align*}
    where the last inequality holds because only the largest eigenvalue of $\mathbf{W}^{T}$ is $1$. In the limit $s \rightarrow \infty$, it follows that
    \begin{align*}
        \|\hat{\mathbf{x}}_{i}^{*} - \bar{\mathbf{x}}^{*}\| = \lim_{s \rightarrow \infty} \|\hat{\mathbf{x}}_{i}^{*} - \hat{\mathbf{W}}^{s+1} \hat{\mathbf{x}}^{*}\| \le \alpha \cdot \frac{\|\nabla F(\hat{\mathbf{x}}^{*})\|}{1-\lambda_{2}}
    \end{align*}
    as claimed.
\end{proof}

\begin{lemmaiii}
    \label{lem: first order guarantee for xhat*}
    Let Assumptions~\ref{ass: Smoothness gradient lipschitz}, \ref{ass: Network} hold. Given $\alpha > 0$, let $\hat{\mathbf{x}}^{*}$ be a local minimizer of $Q_{\alpha}$. Then
    \begin{align*}
        \| \nabla f(\Bar{\mathbf{x}}^{*}) \| \le \alpha \cdot L_{F}^{g} \frac{m\sqrt{m} \|\nabla F(\hat{\mathbf{x}}^{*}) \|}{1-\lambda_{2}},
    \end{align*}
    where $\Bar{\mathbf{x}}^{*} = \frac{1}{m}(\bm{1}_{m} \otimes \mathbf{I}_{n})^{T} \hat{\mathbf{x}}^{*}$ and $0 < \lambda_{2} < 1$ is the second-largest eigenvalue value of $\mathbf{W}$.
\end{lemmaiii}

\begin{proof}
    Since $(\bm{1}_{m} \otimes \mathbf{I}_{n})^{T} \nabla F(\hat{\mathbf{x}}^{*}) = (\bm{1}_{m} \otimes \mathbf{I}_{n})^{T} (\nabla Q_{\alpha}(\hat{\mathbf{x}}^{*}) - \alpha^{-1}(\mathbf{I}_{mn}-\hat{\mathbf{W}})\hat{\mathbf{x}}^{*}) = 0$, by  Assumption \ref{ass: Smoothness gradient lipschitz} and Lemma \ref{lem: consensual bound on xQ*},
    \begin{align*}
        \| \nabla f(\Bar{\mathbf{x}}^{*}) \| &= \|(\bm{1}_{m} \otimes \mathbf{I}_{n})^{T} (\nabla F (\bm{1}_{m} \otimes \Bar{\mathbf{x}}^{*}) - \nabla F (\hat{\mathbf{x}}^{*}))\|\\
        & \le \sqrt{m} L_{F}^{g} \| \bm{1}_{m} \otimes \Bar{\mathbf{x}}^{*} - \hat{\mathbf{x}}^{*} \|\\
        & \le \alpha \cdot L_{F}^{g} \frac{m\sqrt{m} \|\nabla F(\hat{\mathbf{x}}^{*}) \|}{1-\lambda_{2}}
    \end{align*}
    as claimed.
\end{proof}

\begin{lemmaiii}
    \label{lem: second order guarantee for xhat*}
    Let Assumptions~\ref{ass: Smoothness hessian lipschitz}, \ref{ass: Network} hold. Given $\alpha > 0$, let $\hat{\mathbf{x}}^{*}$ be a local minimizer of $Q_{\alpha}$. Then
    \begin{align*}
        \lambda_{\min}(\nabla^{2} f (\Bar{\mathbf{x}}^{*})) \ge -\alpha \cdot L_{F}^{H} \frac{m^{2} \|\nabla F(\hat{\mathbf{x}}^{*}) \|}{1-\lambda_{2}},
    \end{align*}
    where $\Bar{\mathbf{x}}^{*} = \frac{1}{m}(\bm{1}_{m} \otimes \mathbf{I}_{n})^{T} \hat{\mathbf{x}}^{*}$ and $0 < \lambda_{2} < 1$ is the second-largest eigenvalue value of $\mathbf{W}$.
\end{lemmaiii}

\begin{proof}
    Since $(\bm{1}_{m} \otimes \mathbf{I}_{n})^{T} \nabla^{2} F(\hat{\mathbf{x}}^{*})  (\bm{1}_{m} \otimes \mathbf{I}_{n}) = (\bm{1}_{m} \otimes \mathbf{I}_{n})^{T} (\nabla^{2} Q_{\alpha}(\hat{\mathbf{x}}^{*}) - \frac{1}{\alpha}(\mathbf{I}_{mn}-\hat{\mathbf{W}})) (\bm{1}_{m} \otimes \mathbf{I}_{n}) = (\bm{1}_{m} \otimes \mathbf{I}_{n})^{T} (\nabla^{2} Q_{\alpha}(\hat{\mathbf{x}}^{*})) (\bm{1}_{m} \otimes \mathbf{I}_{n}) \succeq 0$, by Assumption \ref{ass: Smoothness hessian lipschitz}, we have
    \begin{align*}
        \nabla^{2} &f(\bar{\mathbf{x}}^{*}) = (\bm{1}_{m} \otimes \mathbf{I}_{n})^{T} \nabla^{2} F(\bm{1}_{m} \otimes \bar{\mathbf{x}}^{*}) (\bm{1}_{m} \otimes \mathbf{I}_{n})\\
        &\succeq (\bm{1}_{m} \otimes \mathbf{I}_{n})^{T} ( \nabla^{2} F(\bm{1}_{m} \otimes \bar{\mathbf{x}}^{*}) - \nabla^{2} F(\hat{\mathbf{x}}^{*}) ) (\bm{1}_{m} \otimes \mathbf{I}_{n})\\
        &\succeq -m L_{F}^{H}  \| \hat{\mathbf{x}}^{*} - \bm{1}_{m} \otimes \bar{\mathbf{x}}^{*} \| \mathbf{I}_{n},
    \end{align*}
    where the last inequality holds because $\textbf{A} \succeq -\|\textbf{A}\| \mathbf{I}_{d}$ for any square matrix $\textbf{A}$, where $\|\mathbf{A}\| = \max_{\mathbf{x}^{T} \mathbf{x}=1} \mathbf{x}^{T} \mathbf{A} \mathbf{x}$. So by Lemma~\ref{lem: consensual bound on xQ*},
    \begin{align*}
        \lambda_{\min}( \nabla^{2} f(\bar{\mathbf{x}}^{*})) \ge - \alpha \cdot L_{F}^{H} \frac{m^{2} \|\nabla F(\hat{\mathbf{x}}^{*}) \|}{1-\lambda_{2}}
    \end{align*}
    as claimed.
\end{proof}

\begin{lemmaiii}
    \label{lem: arbitrarily close}
    Let Assumptions \ref{ass: Strict saddle}, \ref{ass: Smoothness gradient lipschitz}, \ref{ass: Smoothness hessian lipschitz} hold. Then, for any given compact set $\mathcal{X} \subset \mathbb{R}^{n}$,
    \begin{align*}
        \lim_{\alpha\downarrow0} (\sup\{\textup{dist}(\mathbf{x}, \mathcal{X}_{f}^{*}): \mathbf{x} \in {\mathcal{X}^{\alpha}_{f}} \cap \mathcal{X}\}) = 0,
    \end{align*}
    where $\mathcal{X}_{f}^{\alpha} := \{ \mathbf{x}: \|\nabla f(\mathbf{x})\| \le \alpha \cdot c_{1}, \lambda_{\min} (\nabla^{2} f(\mathbf{x})) \ge -\alpha \cdot c_{2}\}$, with
    \begin{align*}
        c_{1} = L_{F}^{g} \frac{m\sqrt{m} \|\nabla F(\hat{\mathbf{x}}^{*}) \|}{1-\lambda_{2}},~
        c_{2} = L_{F}^{H} \frac{m^{2} \|\nabla F(\hat{\mathbf{x}}^{*}) \|}{1-\lambda_{2}}.
    \end{align*}
\end{lemmaiii}

\begin{proof}
    Following the approach used in Lemma 3.8 \cite{daneshmand2020second} to establish a similar result in the absence of the second-order requirement in the definition of $\mathcal{X}_{f}^{\alpha}$, we prove the lemma by contradiction. Suppose
    \begin{align}
    \label{eq: Contradiction hypothesis}
        \inf_{\alpha>0} \{\sup\{\textup{dist}(\mathbf{x}, \mathcal{X}_{f}^{*}): \mathbf{x} \in {\mathcal{X}^{\alpha}_{f}} \cap \mathcal{X}\}\} = d > 0.
    \end{align}
    Then, there exists a sequence $\{\mathbf{x}^{k}\}$ with $\mathbf{x}^{k} \in \mathcal{X}_{f}^{1/k} \cap \mathcal{X}$ and $\textup{dist} (\mathbf{x}^{k}, \mathcal{X}_{f}^{*}) \ge d$ for all $k \in \mathbb{N}^+$. Since $\|\nabla f\|$ and $\lambda_{\min} (\nabla^{2} f)$ are continuous functions (see Assumptions \ref{ass: Smoothness gradient lipschitz}, \ref{ass: Smoothness hessian lipschitz}), $\mathcal{X}_{f}^{1/k}$ is closed. Thus, $\mathcal{X}_{f}^{1/k} \cap \mathcal{X}$ is compact. Since $\{\mathbf{x}^{k}\} \subset \mathcal{X}$, we can find a convergent sub-sequence $\{\mathbf{x}^{t_k}\}$ with limit point $\mathbf{x}^{\infty}$ satisfying $\textup{dist} (\mathbf{x}^{\infty}, \mathcal{X}_{f}^{*}) \ge d$. Since $(\mathcal{X}_{f}^{1/k} \cap \mathcal{X})\supset \mathcal{X}_{f}^{1/(k+1)}$, it follows that, $\mathbf{x}^{\infty} \in \mathcal{X}_{f}^{1/k} \cap \mathcal{X}$ for all $k \in \mathbb{N}^{+}$. This means $\|\nabla f(\mathbf{x}^{\infty})\| \le c_{1} / k $, $\lambda_{\min} (\nabla^{2} f(\mathbf{x}^{\infty})) \ge - c_{2} / k $ for all $k \in \mathbb{N}^+$, implying $\|\nabla f(\mathbf{x}^{\infty})\| = 0$, $\lambda_{\min} (\nabla^{2} f(\mathbf{x}^{\infty})) \ge 0$. By Assumption \ref{ass: Strict saddle}, we have $\lambda_{\min} (\nabla^{2} f(\mathbf{x}^{\infty})) > 0$. Hence $\textup{dist}(\mathbf{x}^{\infty}, \mathcal{X}_{f}^{*}) = 0$, which contradicts the initial hypothesis \eqref{eq: Contradiction hypothesis}.
\end{proof}

By combining Lemmas~\ref{lem: consensual bound on xQ*} through \ref{lem: arbitrarily close}, the following intermediate theorem can be established. It is used to prove Theorem~\ref{the: Second order guarantee} in the next section.

\begin{theoremii}
    \label{the: Local minimizers for Q}
      Let Assumptions~\ref{ass: Strict saddle}, \ref{ass: Smoothness gradient lipschitz}, \ref{ass: Smoothness hessian lipschitz}, \ref{ass: Coercivity}, \ref{ass: Network} hold. Given $\Delta_{1} > 0$, there exists threshold $\bar{\alpha}(\Delta_{1}) > 0$ such that, if $0 < \alpha \le \bar{\alpha}(\Delta_{1})$, and $\hat{\mathbf{x}}^{*}$ is a local minimizer of $Q_{\alpha}$, then
    \begin{align*}
        \textup{dist} (\hat{\mathbf{x}}_{i}^{*}, \mathcal{X}_{f}^{*}) \le \Delta_{1}
    \end{align*}
    for each $i\in\mathcal{V}$.
\end{theoremii}

\begin{proof}
    Given $\alpha > 0$, by the triangle inequality, $\textup{dist} (\hat{\mathbf{x}}_{i}^{*}, \mathcal{X}_{f}^{*}) \le \|\hat{\mathbf{x}}_{i}^{*} - \Bar{\mathbf{x}}^{*}\| + \textup{dist} (\Bar{\mathbf{x}}^{*}, \mathcal{X}_{f}^{*})$, where $\Bar{\mathbf{x}}^{*} = \frac{1}{m}(\bm{1}_{m} \otimes \mathbf{I}_{n})^{T} \hat{\mathbf{x}}^{*}$. By coercivity and properness of each $f_i$ (see Assumption \ref{ass: Coercivity}), $F$ is coercive and proper. Therefore, $\hat{\mathcal{X}}^{*}_{Q_{\alpha}}$ is bounded, and there exists an upper bound $G > 0$ such that for all $\hat{\mathbf{x}}^{*} \in \hat{\mathcal{X}}^{*}_{Q_{\alpha}}$, $\|\nabla F(\hat{\mathbf{x}}^{*})\| \le G$. By Lemma~\ref{lem: consensual bound on xQ*}, if 
    \begin{align*}
        0< \alpha \le \bar{\alpha}_{1}(\Delta_{1}) := \frac{\Delta_{1}(1-\lambda_{2})}{2G}
    \end{align*}
    and $\hat{\mathbf{x}}^{*} \in \hat{\mathcal{X}}^{*}_{Q_{\alpha}}$ defined in \eqref{eq: minimizer set}, then $\|\hat{\mathbf{x}}_{i}^{*} - \Bar{\mathbf{x}}^{*}\| \le \Delta_{1}/2$ holds for each $i \in \mathcal{V}$. Now, note that $\bar{\mathbf{x}}^{*}\in\mathcal{X}_{f}^{\alpha}$ in view of Lemmas~\ref{lem: first order guarantee for xhat*} and~\ref{lem: second order guarantee for xhat*}, with $\mathcal{X}_{f}^{\alpha}$ as defined in Lemma~\ref{lem: arbitrarily close}. As such, by application of Lemma~\ref{lem: arbitrarily close} with $\mathcal{X}=\{\bar{\mathbf{x}}^{*}\}$, there exists $\bar{\alpha}_{2}(\Delta_{1})>0$ such that if $0 < \alpha \le \bar{\alpha}_{2}(\Delta_{1})$, then $\textup{dist} (\Bar{\mathbf{x}}^{*}, \mathcal{X}_{f}^{*}) \le \Delta_{1}/2$ holds. Therefore, if 
    \begin{align*}
        0 < \alpha \le \bar{\alpha}(\Delta_{1}) := \min\{\bar{\alpha}_{1}(\Delta_{1}),\bar{\alpha}_{2}(\Delta_{1})\},
    \end{align*}
    then $\textup{dist} (\hat{\mathbf{x}}_{i}^{*}, \mathcal{X}_{f}^{*}) \le \Delta_{1}$ as claimed.
\end{proof}


\section{Method and Main Results}
\label{sec: Method and Main Results}
In this section, a \textbf{N}oisy \textbf{D}istributed \textbf{G}radient \textbf{D}escent algorithm (see Algorithm \ref{alg: Noisy dgd}), a variant of the fixed step-size \textbf{DGD} algorithm, is formulated. The main analysis result, also formulated in this section, establishes the second-order properties of the \textbf{NDGD} algorithm. The key idea is to add random noise to the distributed gradient descent directions at each iteration. The required properties of the noise $\xi_{i}^{k}$ in Algorithm \ref{alg: Noisy dgd} is presented in Theorem \ref{the: Second order guarantee}. 


\begin{algorithm}
    \begin{algorithmic}
    \State \textbf{Initialization};
        \For{$k = 0,1,\cdots$}
            \For{$i = 0,1,\cdots,m$}
                \State Sample i.i.d $\xi_{i}^{k}$;
                \State $\hat{\mathbf{x}}_{i}^{k+1} = \sum_{j=1}^m \mathbf{W}_{ij} \hat{\mathbf{x}}_{j}^{k} - \alpha (\nabla f_i(\hat{\mathbf{x}}_{i}^{k}) + \xi_{i}^{k})$;
            \EndFor
        \EndFor
    \end{algorithmic}
\caption{\textbf{N}oisy \textbf{D}istributed \textbf{G}radient \textbf{D}escent (\textbf{NDGD})}
\label{alg: Noisy dgd}
\end{algorithm}

Recall that $\mathcal{X}^{*}_f$ and $\hat{\mathcal{X}}^{*}_{Q_{\alpha}}$ denote the set of local minimizers of $f$ and $Q_{\alpha}$ respectively, as per \eqref{eq: minimizer set}. Given $\epsilon > 0$, $\gamma > 0$, $\mu > 0$, $\delta > 0$, and $\alpha > 0$, define
\begin{align}
\label{eq: regularity set of f}
    \begin{aligned}
        \mathcal{L}_{\alpha,\epsilon}^{1} : = \{\hat{\mathbf{x}}:&~\|\nabla F(\hat{\mathbf{x}}) + \alpha^{-1}(\mathbf{I}_{mn}-\hat{\mathbf{W}}) \hat{\mathbf{x}}\| \ge \epsilon\},\\
        \mathcal{L}_{\alpha,\gamma}^{2} : = \{\hat{\mathbf{x}}:&~\lambda_{\min}(\nabla^2 F(\hat{\mathbf{x}})) \le -\gamma - \alpha^{-1}\},\\
        \mathcal{L}_{\alpha,\mu,\delta}^{3} : = \{\hat{\mathbf{x}}:&~\lambda_{\min}(\nabla^2 F(\hat{\mathbf{x}})) \ge \mu,~\textup{dist} (\hat{\mathbf{x}}, \hat{\mathcal{X}}^{\prime}) \le \delta\},
    \end{aligned}
\end{align}
where $\hat{\mathcal{X}}^{\prime} := \{\hat{\mathbf{x}} : \|\nabla F(\hat{\mathbf{x}}) + \alpha^{-1}(\mathbf{I}_{mn}-\hat{\mathbf{W}}) \hat{\mathbf{x}}\| = 0,~\lambda_{\min}(\nabla^2 F(\hat{\mathbf{x}})) > 0\}$. The next theorem, which establishes second-order properties of the \textbf{NDGD} algorithm, is the main result of the paper; a proof is given in Section~\ref{sec: Proof of Main Theorem}. Note that this result focuses on the dependency on the given step-size $\alpha$ and confidence parameter $\zeta$, hiding the factors that have polynomial dependence on all other parameters (including $\Delta_1$, $\epsilon$, $\gamma$, $\mu$, $\delta$ and $\sigma$ ).

\begin{theoremii}
\label{the: Second order guarantee}
    Suppose Assumptions~\ref{ass: Strict saddle}, \ref{ass: Smoothness gradient lipschitz}, \ref{ass: Smoothness hessian lipschitz}, \ref{ass: Coercivity}, \ref{ass: Network} hold. Given $\Delta_{1} > 0$ and $0 < \zeta < 1$, also suppose the following:
    \begin{enumerate}
    \item \label{ite: supposition 1} There exist $\epsilon>0$, $\gamma \in (0,L_{F}^{g}]$, $\mu \in (0,L_{F}^{g}]$, $\delta>0$ and $\alpha \in (0,\hat{\alpha}(\Delta_{1},\zeta)]$ such that 
    \begin{align*}
        \mathcal{L}_{\alpha,\epsilon}^{1} \cup \mathcal{L}_{\alpha,\gamma}^{2} \cup \mathcal{L}_{\alpha,\mu,\delta}^{3} = (\mathbb{R}^{n})^{m},
    \end{align*}
    where $\hat{\alpha}(\Delta_{1},\zeta) :=$
    \begin{align*}
        \min\{\bar{\alpha}(\Delta_{1}), \frac{\sqrt{2}-1}{L_{F}^{g}}, \frac{\lambda_{\min}(\mathbf{W})}{L_{F}^{g} \cdot \max\{1,\log(\zeta^{-1})\}}\}>0
    \end{align*}
    with $\bar{\alpha}(\Delta_1)$ as per
    Theorem \ref{the: Local minimizers for Q};
    \item \label{ite: supposition 2} the random perturbation $\xi_i^k$ at step $k>0$ is i.i.d. and zero mean with variance
    \begin{align*}
        \sigma^2 \leq \sigma_{\max}^2(\epsilon) := \frac{\lambda_{\min}(\mathbf{W})\epsilon^2}{mn};
    \end{align*}
    \item the generated sequence $\{Q_{\alpha}(\hat{\mathbf{x}}^k)\}$ is bounded.
    \end{enumerate}
    Then, with probability at least $1-\zeta$, after $K(\alpha,\zeta)=\mathcal{O}(\alpha^{-2}\log\zeta^{-1})$ iterations, Algorithm \ref{alg: Noisy dgd} reaches a point $\hat{\mathbf{x}}^{K(\alpha,\zeta)} \in (\mathbb{R}^{n})^{m}$ that is $\Delta_{2}(\alpha,\zeta)$-close to $\mathcal{X}_{Q_{\alpha}}^{*}$, where $\Delta_{2}(\alpha,\zeta) = \mathcal{O}(\sqrt{\alpha\log(\alpha^{-1}\zeta^{-1})})$. Moreover, $\hat{\mathbf{x}}^{*}=\inf_{\hat{\mathbf{x}} \in \mathcal{X}_{Q_{\alpha}}^{*}}\|\hat{\mathbf{x}}_{i}^{K(\alpha,\zeta)} -\hat{\mathbf{x}}\|$ is such that $\hat{\mathbf{x}}_{i}^{*}$ is $\Delta_{1}$-close to $\mathcal{X}_{f}^{*}$, whereby for all $i \in \mathcal{V}$,
    \begin{align*}
        \mathrm{dist}(\hat{\mathbf{x}}_{i}^{K(\alpha,\zeta)},\mathcal{X}_{f}^{*}) \leq \Delta_{1} + \Delta_{2}(\alpha,\zeta).
    \end{align*}
    
\end{theoremii}

\begin{remark}
    Intuitively, condition \textup{\ref{ite: supposition 1})} in Theorem \ref{the: Second order guarantee} requires that for $Q_{\alpha}$, all points with a small norm of gradient, they either possess a sufficient descent direction or reside within a neighborhood of a local minimizer of $Q_{\alpha}$, where local strong convexity is present.
\end{remark}

\begin{remark}
    For condition \textup{\ref{ite: supposition 2})} in Theorem \ref{the: Second order guarantee}, one way to generate such i.i.d noise is to sample $\xi_i^{k}$ uniformly from an $n$-dimensional sphere with the radius $r$. This ensures $\mathbb{E}(\xi_i^k) = \bm{0}$, $\mathbb{E}(\xi_i^k (\xi_i^k)^T) = (r^2/n) \mathbf{I}_n$, and $\|\xi_i^k\| \le r$ for all $i \in \mathcal{V}$ and $k \in \mathbb{N}$. By choosing $r^2 \le n\sigma_{\max}^2(\epsilon)$, we have $\mathbb{E}(\xi^k (\xi^k)^T) = (r^2/n) \mathbf{I}_{mn} \preceq \sigma_{\max}^2(\epsilon) \mathbf{I}_{mn}$.
\end{remark}

Second-order guarantees of \textbf{DGD} have been studied in \cite{daneshmand2020second} and \cite{swenson2022distributed} based on the almost sure non-convergence to saddle points with random initialization. In this paper, we propose to use random perturbations to actively evade saddle points. The second-order guarantees of \textbf{NDGD} stated in Theorem \ref{the: Second order guarantee} do not require any additional initialization conditions. Second-order guarantees of the stochastic variant of \textbf{DGD} have been studied in \cite{vlaski2021distributed1} and \cite{vlaski2021distributed2}, although they only show the convergence to an approximate second-order stationary point. Here, an upper bound is given for the distance between the iterate at each agent and the set of local minimizers after a sufficient number of iterations.

\section{Proof of Theorem \ref{the: Second order guarantee}}
\label{sec: Proof of Main Theorem}
A proof of Theorem \ref{the: Second order guarantee} is provided in Section~\ref{sec: Proof of Main Theorem}. First, we consider three different cases to study the behavior of \textbf{NDGD} for three different cases, in line with the development of the related result in \cite{ge2015escaping} for centralized gradient descent. i) large in norm $\nabla Q_{\alpha}(\hat{\mathbf{x}}^{k})$ (see Lemma~\ref{lem: Decrease for large gradient}); ii) sufficiently negative $\lambda_{\min}(\nabla^{2} Q_{\alpha}(\hat{\mathbf{x}}^{k}))$ (see Lemma~\ref{lem: Decrease for strict saddle}); and iii) $\hat{\mathbf{x}}^{k}$ in a neighborhood of the local minimizers of $Q_{\alpha}$ with local strong convexity (see Lemma~\ref{lem: Stay close to minimizers}). Combining the outcome of this with Theorem \ref{the: Local minimizers for Q}, we then prove that with probability at least $1-\zeta$, after $K(\alpha,\zeta)$ iterations, \textbf{NDGD} yields a point of which the component at each agent is $\Delta_{1} + \Delta_{2}(\alpha,\zeta)$-close to some local minimizer of $f$.

\subsection{Behavior of \textbf{NDGD} for three different cases}
    The following lemmas rely on the proofs of Lemma 16 and Lemma 17 in \cite{ge2015escaping}. Given $\epsilon > 0$, $\gamma > 0$, $\mu > 0$, $\delta > 0$, and $\alpha > 0$, define
    \begin{align}
    \label{eq: regularity set of Q}
        \begin{aligned}
            \mathcal{I}_{\alpha,\epsilon}^{1} : = \{\hat{\mathbf{x}}:&~\|\nabla Q_{\alpha} (\hat{\mathbf{x}})\| \ge \epsilon\},\\
            \mathcal{I}_{\alpha,\gamma}^{2} : = \{\hat{\mathbf{x}}:&~\Lambda_{\alpha} (\hat{\mathbf{x}}) \le -\gamma\},\\
            \mathcal{I}_{\alpha,\mu,\delta}^{3} : = \{\hat{\mathbf{x}}:&~\Lambda_{\alpha} (\hat{\mathbf{x}}) \ge \mu,~\textup{dist} (\hat{\mathbf{x}}, \hat{\mathcal{X}}^{*}_{Q_{\alpha}}) \le \delta\},
        \end{aligned}
    \end{align}
    where $\Lambda_{\alpha} (\hat{\mathbf{x}}) = \lambda_{\min}(\nabla^{2} Q_{\alpha}(\hat{\mathbf{x}}))$. 

    We first analyze the behavior of the \textbf{NDGD} algorithm in the case that $\hat{\mathbf{x}}^{k} \in \mathcal{I}_{\alpha,\epsilon}^{1}$. Intuitively, when the norm of $\nabla Q_{\alpha}(\hat{\mathbf{x}})$ is large enough, the expectation of the function value decreases by a certain amount after one iteration. 
    \begin{lemmaiii}
    \label{lem: Decrease for large gradient}
        Let Assumption~\ref{ass: Smoothness gradient lipschitz} hold. Given $\epsilon > 0$, suppose the random perturbation $\xi^{k}_i$ in Algorithm~\ref{alg: Noisy dgd} is i.i.d. and zero mean with variance $\sigma^2 \le \sigma_{\max}^2(\epsilon) := (\lambda_{\min}(\mathbf{W})\epsilon^2)/(mn)$.
        Then given $0 < \alpha \le 1/L_{F}^{g}$,
        for any $\hat{\mathbf{x}}^{k}$ such that $\|\nabla Q_{\alpha} (\hat{\mathbf{x}}^{k})\| \ge \epsilon$, after one iteration,
        \begin{align*}
            \mathbb{E} [Q_{\alpha} (\hat{\mathbf{x}}^{k+1}) ~|~ \hat{\mathbf{x}}^{k}] - Q_{\alpha} (\hat{\mathbf{x}}^{k}) \le -l_{1}(\alpha),
        \end{align*}
        where $l_{1}(\alpha) = \Omega(\alpha)$.
    \end{lemmaiii}

    \begin{proof}
        Since $\mathbf{W}$ is symmetric, doubly stochastic and strictly diagonally dominant, it is positive definite by the Gershgorin disk theorem. Given $\sigma^{2} \le (\lambda_{\min}(\mathbf{W})\epsilon^2) / (mn)$, note that $\sqrt{mn\sigma^{2} / \lambda_{\min}(\mathbf{W})} < \epsilon$. Since $Q_{\alpha}$ has $L_{Q_{\alpha}}^{g}$-Lipschitz continuous gradient, using Taylor's theorem gives
        \begin{align*}
            \mathbb{E} [Q&_{\alpha} (\hat{\mathbf{x}}^{k+1}) ~|~ \hat{\mathbf{x}}^{k}] - Q_{\alpha} (\hat{\mathbf{x}}^{k}) \\
            =&~\mathbb{E} [\nabla Q_{\alpha} (\hat{\mathbf{x}}^{k})^{T} (\hat{\mathbf{x}}^{k+1} - \hat{\mathbf{x}}^{k}) + (\hat{\mathbf{x}}^{k+1} - \hat{\mathbf{x}}^{k})^{T}(\int^{1}_{0} (1-t)\\
            &~\quad\nabla^2 Q_{\alpha} (\hat{\mathbf{x}}^{k} + t(\hat{\mathbf{x}}^{k+1} - \hat{\mathbf{x}}^{k})) \textup{d}t) (\hat{\mathbf{x}}^{k+1} - \hat{\mathbf{x}}^{k}) ~|~ \hat{\mathbf{x}}^{k}]\\
            \le&~\nabla Q_{\alpha} (\hat{\mathbf{x}}^{k})^{T} \mathbb{E} [\hat{\mathbf{x}}^{k+1} - \hat{\mathbf{x}}^{k} ~|~ \hat{\mathbf{x}}^{k}]\\ 
            &~ \quad+\frac{L_{Q_{\alpha}}^{g}}{2}~\mathbb{E} [\|\hat{\mathbf{x}}^{k+1} - \hat{\mathbf{x}}^{k}\|^{2} ~|~ \hat{\mathbf{x}}^{k}]\\
            =&~\nabla Q_{\alpha} (\hat{\mathbf{x}}^{k})^{T} \mathbb{E} [-\alpha (\nabla Q_{\alpha} (\hat{\mathbf{x}}^{k}) + \xi^{k}) ~|~ \hat{\mathbf{x}}^{k}]\\
            &~ \quad+\frac{L_{Q_{\alpha}}^{g}}{2}~\mathbb{E} [\|-\alpha (\nabla Q_{\alpha} (\hat{\mathbf{x}}^{k}) + \xi^{k})\|^{2} ~|~ \hat{\mathbf{x}}^{k}]\\
            =&~(-\alpha + \frac{L_{Q_{\alpha}}^{g}}{2} \alpha^{2}) \|\nabla Q_{\alpha} (\hat{\mathbf{x}}^{k})\|^{2} + \frac{L_{Q_{\alpha}}^{g}}{2}\alpha^{2}mn\sigma^{2}\\
            =&~(-\frac{1+\lambda_{\min}(\mathbf{W})}{2} \alpha + \frac{L_{F}^{g}}{2} \alpha^{2}) \|\nabla Q_{\alpha} (\hat{\mathbf{x}}^{k})\|^{2}\\
            &~ \quad+(\frac{1-\lambda_{\min}(\mathbf{W})}{2} \alpha + \frac{L_{F}^{g}}{2} \alpha^{2})mn\sigma^{2}.
        \end{align*}
        As such, choosing $0 < \alpha \le 1 / L_{F}^{g}$ gives
        \begin{align*}
            \mathbb{E} [Q&_{\alpha} (\hat{\mathbf{x}}^{k+1}) ~|~ \hat{\mathbf{x}}^{k}] - Q_{\alpha} (\hat{\mathbf{x}}^{k})\\
            \le&-\frac{\lambda_{\min}(\mathbf{W})}{2} \alpha \|\nabla Q_{\alpha} (\hat{\mathbf{x}}^{k})\|^{2} + \frac{2 - \lambda_{\min}(\mathbf{W})}{2}\alpha mn\sigma^{2}\\
            \\[-3pt]
            \le&-\frac{\lambda_{\min}(\mathbf{W})}{2}\alpha mn\sigma^{2}
        \end{align*}
        as claimed.
    \end{proof}

    Next, we analyze the behavior of the \textbf{NDGD} algorithm in the case that $\hat{\mathbf{x}}^{k} \in \mathcal{I}_{\alpha,\gamma}^{2}$. Intuitively, for $\hat{\mathbf{x}}^{k}$ with small gradient and sufficiently negative $\lambda_{\min}(\nabla^{2} Q_{\alpha}(\hat{\mathbf{x}}))$, there exists upper bound $T_{\max}(\alpha)$ such that the expectation of the function value decreases by a certain amount after $T \le T_{\max}(\alpha)$ iterations. 
    \begin{lemmaiii}
    \label{lem: Decrease for strict saddle}
        Let Assumptions \ref{ass: Smoothness gradient lipschitz}, \ref{ass: Smoothness hessian lipschitz} hold. Let $\gamma \in (0,L_{F}^{g}]$. Given $\epsilon > 0$, suppose the random perturbation $\xi^{k}_i$ in Algorithm~\ref{alg: Noisy dgd} is i.i.d. and zero mean with variance $\sigma^2 \le \sigma_{\max}^2(\epsilon) := (\lambda_{\min}(\mathbf{W})\epsilon^2)/(mn)$.
        Further, given $0 < \alpha \le (\sqrt{2}-1)/L_{F}^{g}$,
        suppose that the generated sequence $\{Q_{\alpha}(\hat{\mathbf{x}}^{k})\}$ is bounded. Then, for any $\hat{\mathbf{x}}^{k}$ with $\|\nabla Q_{\alpha} (\hat{\mathbf{x}}^{k})\| < \epsilon$ and $\lambda_{\min}(\nabla^{2} Q_{\alpha}(\hat{\mathbf{x}})) \le -\gamma$, there exists a number of steps $T(\hat{\mathbf{x}}^{k}) > 0$ such that
        \begin{align*}
            \mathbb{E} [Q_{\alpha} (\hat{\mathbf{x}}^{k+T(\hat{\mathbf{x}}^{k})}) ~|~ \hat{\mathbf{x}}^{k}] - Q_{\alpha} (\hat{\mathbf{x}}^{k}) \le -l_{2}(\alpha),
        \end{align*}
        where $l_{2}(\alpha) = \Omega(\alpha)$. The number of steps $T(\hat{\mathbf{x}}^{k})$ has a fixed upper bound $T_{\max}(\alpha)$ that is independent of $\hat{\mathbf{x}}^{k}$, i.e., $T(\hat{\mathbf{x}}^{k}) \le T_{\max}(\alpha) = \mathcal{O}(\alpha^{-1})$ for all $\hat{\mathbf{x}}^{k}$.
    \end{lemmaiii}
    \begin{proof}
        Given $\sigma^{2} \le (\lambda_{\min}(\mathbf{W})\epsilon^2) / (mn)$, note that $\sqrt{mn\sigma^{2} / \lambda_{\min}(\mathbf{W})} < \epsilon$. Thus, choosing $0 < \alpha \le (\sqrt{2}-1) / L_{F}^{g} \le (2mn) / L_{F}^{g} \le (2mn) / \gamma$, the result holds as shown in the proof of Lemma 17 in \cite{ge2015escaping}.
    \end{proof}

    Finally, we analyze the behavior of the algorithm in the case that $\hat{\mathbf{x}}^{k} \in \mathcal{I}_{\alpha,\mu,\delta}^{3}$. Intuitively, when the iterate $\hat{\mathbf{x}}^k$ is close enough to a local minimizer, with high probability subsequent iterates do not leave the neighborhood. 
    \begin{lemmaiii}
    \label{lem: Stay close to minimizers}
        Let Assumptions \ref{ass: Smoothness gradient lipschitz} hold. Let $\mu \in (0,L_{F}^{g}]$. Given $\epsilon > 0$, suppose the random perturbation $\xi^{k}_i$ in Algorithm~\ref{alg: Noisy dgd} is i.i.d. and zero mean with variance $\sigma^2 \le \sigma_{\max}^2(\epsilon) := (\lambda_{\min}(\mathbf{W})\epsilon^2)/(mn)$.
        Further, given $\delta>0$, $0 < \alpha \le (\lambda_{\min}(\mathbf{W}))/(L_{F}^{g} \cdot \max\{1,\log(\zeta^{-1})\})$
        and local minimizer $\hat{\mathbf{x}}^{*} \in \mathcal{X}_{Q_{\alpha}}^{*}$, suppose $\lambda_{\min}(\nabla^{2} Q_{\alpha}(\hat{\mathbf{x}})) \ge \mu$ for all $\hat{\mathbf{x}}$ such that $\|\hat{\mathbf{x}}-\hat{\mathbf{x}}^{*}\|<\delta$. Then, there exists $\delta_{1}(\alpha)=\mathcal{O}(\sqrt{\alpha})\in[0,\delta)$ such that, for any $\hat{\mathbf{x}}^{k}$ with $\|\hat{\mathbf{x}}^{k}-\hat{\mathbf{x}}^{*}\|<\delta_{1}(\alpha)$, with probability at least $1-\zeta/2$,
        \begin{align*}
            \|\hat{\mathbf{x}}^{k+s} - \hat{\mathbf{x}}^{*}\| \le \Delta_{2}(\alpha,\zeta)
        \end{align*}
        for all $s \le K(\alpha,\zeta) = \mathcal{O}(\alpha^{-2}\log(\zeta^{-1}))$, where $\Delta_{2}(\alpha,\zeta) = \mathcal{O}(\sqrt{\alpha \log(\alpha^{-1} \zeta^{-1})}) < \delta$.
    \end{lemmaiii}

    \begin{proof}
        Given $0 < \alpha \le \lambda_{\min}(\mathbf{W}) / L_{F}^{g}$, note that $\alpha \le 1 / L_{Q_{\alpha}}^{g}$ because the eigenvalues of the symmetric, doubly stochastic, diagonally dominant, and thus positive definite matrix $\mathbf{W}$, is contained in the interval $(0,1]$. By $\lambda_{\min}(\nabla^{2} Q_{\alpha}(\hat{\mathbf{x}}^{k})) \ge \mu > 0$, we have the local strong convexity with modulus $\mu$. Since $Q_{\alpha}$ also has Lipschitz gradient, by Theorem 2.1.12 \cite{nesterov2003introductory},
        \begin{multline*}
            \nabla Q_{\alpha}(\hat{\mathbf{x}})^{T} (\hat{\mathbf{x}} - \hat{\mathbf{x}}^{*}) \ge \frac{L_{Q_{\alpha}}^{g} \mu}{L_{Q_{\alpha}}^{g} + \mu} \|\hat{\mathbf{x}} - \hat{\mathbf{x}}^{*}\|^{2}\\ 
             + \frac{1}{L_{Q_{\alpha}}^{g} + \mu}\|\nabla Q_{\alpha}(\hat{\mathbf{x}}) - \nabla Q_{\alpha}(\hat{\mathbf{x}}^{*})\|^{2}.
        \end{multline*}
        Therefore,
        \begin{align*}
            \mathbb{E}[\|&\hat{\mathbf{x}}^{k+1} - \hat{\mathbf{x}}^{*}\|^{2} ~|~ \hat{\mathbf{x}}^{k}]\\ 
            =&~\mathbb{E}[\|\hat{\mathbf{x}}^{k} - \alpha(\nabla Q_{\alpha}(\hat{\mathbf{x}}^{k}) + \xi^{k}) - \hat{\mathbf{x}}^{*}\|^{2} ~|~ \hat{\mathbf{x}}^{k}]\\
            \le&~\|(\hat{\mathbf{x}}^{k} - \hat{\mathbf{x}}^{*})\|^{2} + \alpha^{2} \|\nabla Q_{\alpha}(\hat{\mathbf{x}}^{k}) - \nabla Q_{\alpha}(\hat{\mathbf{x}}^{*})\|^{2}\\
            &~ \quad- 2\alpha\nabla Q_{\alpha}(\hat{\mathbf{x}}^{k})^{T} (\hat{\mathbf{x}}^{k} - \hat{\mathbf{x}}^{*}) + \alpha^{2} mn \sigma^{2}\\
            \le&~\|(\hat{\mathbf{x}}^{k} - \hat{\mathbf{x}}^{*})\|^{2} + \alpha^{2} \|\nabla Q_{\alpha}(\hat{\mathbf{x}}^{k}) - \nabla Q_{\alpha}(\hat{\mathbf{x}}^{*})\|^{2}\\
            &~ \quad- 2\alpha(\frac{1}{L_{Q_{\alpha}}^{g} + \mu}\|\nabla Q_{\alpha}(\hat{\mathbf{x}}^{k}) - \nabla Q_{\alpha}(\hat{\mathbf{x}}^{*})\|^{2}\\ 
            &~ \quad+ \frac{L_{Q_{\alpha}}^{g} \mu}{L_{Q_{\alpha}}^{g} + \mu}\|\hat{\mathbf{x}}^{k} - \hat{\mathbf{x}}^{*}\|^{2}) + \alpha^{2} mn \sigma^{2}\\
            \quad=&~(1 - 2\alpha\frac{L_{Q_{\alpha}}^{g} \mu}{L_{Q_{\alpha}}^{g} + \mu}) \|\hat{\mathbf{x}}^{k} - \hat{\mathbf{x}}^{*}\|^{2} + (\alpha^{2} - 2\alpha\frac{1}{L_{Q_{\alpha}}^{g} + \mu})\\
            &~\quad\|\nabla Q_{\alpha}(\hat{\mathbf{x}}^{k}) - \nabla Q_{\alpha}(\hat{\mathbf{x}}^{*})\|^{2} + \alpha^{2} mn \sigma^{2}.
        \end{align*}
        Since $\alpha \le 1 / L_{Q_{\alpha}}^{g}$, note that $\alpha - 2 / (L_{Q_{\alpha}}^{g} + \mu) < 0$ with $\mu \le L_{Q_{\alpha}}^{g}$. Then, 
        \begin{align*}
            \mathbb{E}[\|\hat{\mathbf{x}}^{k} - \hat{\mathbf{x}}^{*}\|^{2}] \le&~(1-\frac{2\alpha \cdot L_{Q_{\alpha}}^{g} \mu}{L_{Q_{\alpha}}^{g} + \mu} + \alpha^{2}\mu^{2} - \frac{2\alpha \cdot \mu^{2}}{L_{Q_{\alpha}}^{g} + \mu})\\
            &~\quad\|\hat{\mathbf{x}}^{k-1} - \hat{\mathbf{x}}^{*}\|^{2} + \alpha^{2} mn \sigma^{2}\\
            \le&~(1-\alpha\mu)^{2} \|\hat{\mathbf{x}}^{k-1} - \hat{\mathbf{x}}^{*}\|^{2} + \alpha^{2} mn \sigma^{2}.
        \end{align*}
        Thus, choosing $\alpha \le 1 / L_{Q_{\alpha}}^{g} \le 1 / L_{F}^{g} \le \mu^{-1}$, the result holds as shown in the proof of Lemma 16 in \cite{ge2015escaping}.
    \end{proof}

\subsection{Main proof}
\begin{proof}
    The main proof includes two steps: i) it is shown that three sets defined in \eqref{eq: regularity set of Q} cover all possible points with respect to $Q_\alpha$; ii) it is shown that the upper bound of the decrease in $Q_{\alpha}$ can be used to derive a lower bound for the probability that the $K(\alpha,\zeta)$-th update at each agent is close to a local minimizer of $f$.

    \textbf{Step 1.} By the supposition in Theorem \ref{the: Second order guarantee}, given $\Delta_{1} > 0$ and $0 < \zeta < 1$, there exist $\epsilon > 0$, $0 < \gamma \le L_{F}^g$, $0 < \mu \le L_{F}^g$, $\delta > 0$, and
     \begin{align*}
         0 < \alpha \le \min\{\bar{\alpha}(\Delta_{1}), \frac{\sqrt{2}-1}{L_{F}^{g}}, \frac{\lambda_{\min}(\mathbf{W})}{L_{F}^{g} \cdot \max\{1,\log(\zeta^{-1})\}}\},
     \end{align*}
     such that $\mathcal{L}_{\alpha,\epsilon}^{1} \cup \mathcal{L}_{\alpha,\gamma}^{2} \cup \mathcal{L}_{\alpha,\mu,\delta}^{3} = (\mathbb{R}^{n})^{m}$, with respect to \eqref{eq: regularity set of f}, and thus $\mathcal{L}_{\alpha,\mu,\delta}^{3} \supseteq (\mathcal{L}_{\alpha,\epsilon}^{1}\cup \mathcal{L}_{\alpha,\gamma}^{2})^c$, where the superscript $c$ denote set complement. If $\hat{\mathbf{x}} \in \mathcal{L}_{\alpha,\epsilon}^{1}$,
    \begin{align*}
        \|\nabla Q_{\alpha} (\hat{\mathbf{x}})\| &= \|\nabla F(\hat{\mathbf{x}}) + \frac{1}{\alpha} (\mathbf{I}_{mn}-\hat{\mathbf{W}})\hat{\mathbf{x}}\| \ge \epsilon;
    \end{align*}
    if $\hat{\mathbf{x}} \in \mathcal{L}_{\alpha,\gamma}^{2}$, then by Weyl's inequality,
    \begin{align*}
        \lambda_{\min}(\nabla^{2} Q_{\alpha} (\hat{\mathbf{x}})) = \lambda_{\min}(\nabla^{2} F(\hat{\mathbf{x}}) + \frac{1}{\alpha} (\mathbf{I}_{mn}-\hat{\mathbf{W}})) \le  -\gamma;
    \end{align*}
    if $\hat{\mathbf{x}} \in \mathcal{L}_{\alpha,\mu,\delta}^{3}$, then again by Weyl's inequality,
    \begin{align*}
        \lambda_{\min}(\nabla^{2} Q_{\alpha} (\hat{\mathbf{x}})) = \lambda_{\min}(\nabla^{2} F(\hat{\mathbf{x}}) + \frac{1}{\alpha} (\mathbf{I}_{mn}-\hat{\mathbf{W}})) \ge \mu,
    \end{align*}
    and $\textup{dist} (\hat{\mathbf{x}}, \hat{\mathcal{X}}^{*}_{Q_{\alpha}}) \le \delta$.
    Therefore,
    $\mathcal{L}_{\alpha,\epsilon}^{1} \subseteq \mathcal{I}_{\alpha,\epsilon}^{1}$, $\mathcal{L}_{\alpha,\gamma}^{2} \subseteq \mathcal{I}_{\alpha,\gamma}^{2}$, $\mathcal{L}_{\alpha,\mu,\delta}^{3} \subseteq \mathcal{I}_{\alpha,\mu,\delta}^{3}$,
    whereby 
    \begin{align*}
        \mathcal{I}_{\alpha,\epsilon}^{1} \cup \mathcal{I}_{\alpha,\gamma}^{2} \cup \mathcal{I}_{\alpha,\mu,\delta}^{3} = (\mathbb{R}^{n})^{m},~\mathcal{I}_{\alpha,\mu,\delta}^{3} \supseteq (\mathcal{I}_{\alpha,\epsilon}^{1}\cup \mathcal{I}_{\alpha,\gamma}^{2})^c.
    \end{align*}
    

    
    \textbf{Step 2.} Define stochastic process $\{\kappa_{i}\} \subset \mathbb{N}$ as
    \begin{align}
    \label{eq: stochastic process}
        \kappa_{i} := \begin{cases}0, & i=0\\ 
        \kappa_{i-1} + 1, & \hat{\mathbf{x}}^{\kappa_{i-1}} \in \mathcal{I}_{\alpha,\epsilon}^{1} \cup \mathcal{I}_{\alpha,\mu,\delta}^{3}\\
        \kappa_{i-1} + T(\hat{\mathbf{x}}^{\kappa_{i-1}}), & \hat{\mathbf{x}}^{\kappa_{i-1}} \in \mathcal{I}_{\alpha,\gamma}^{2}\end{cases},
    \end{align} 
    where $T(\hat{\mathbf{x}}) \le T_{\max}(\alpha)  = \Tilde{\mathcal{O}}(\alpha^{-1})$ for all $\hat{\mathbf{x}}$ as per Lemma \ref{lem: Decrease for strict saddle}. By Lemma \ref{lem: Decrease for large gradient} and Lemma \ref{lem: Decrease for strict saddle}, $Q_{\alpha}$ decreases by a certain amount after a certain number of iterations for $\hat{\mathbf{x}} \in \mathcal{I}_{\alpha,\epsilon}^{1}$, and $\hat{\mathbf{x}} \in \mathcal{I}_{\alpha,\gamma}^{2}$, respectively, as follows
    \begin{align}
    \label{eq: expectation over L1 and L2}
        \begin{aligned}
            \mathbb{E}[Q_{\alpha} (\hat{\mathbf{x}}^{\kappa_{i+1}}) - Q_{\alpha} (\hat{\mathbf{x}}^{\kappa_{i}}) ~|~ \hat{\mathbf{x}}^{\kappa_{i}} \in \mathcal{I}_{\alpha,\epsilon}^{1}] \le -l_{1}(\alpha),\\
            \mathbb{E}[Q_{\alpha} (\hat{\mathbf{x}}^{\kappa_{i+1}}) - Q_{\alpha} (\hat{\mathbf{x}}^{\kappa_{i}}) ~|~ \hat{\mathbf{x}}^{\kappa_{i}} \in \mathcal{I}_{\alpha,\gamma}^{2}] \le -l_{2}(\alpha),
        \end{aligned}
    \end{align}
    where $l_{1}(\alpha) = \Omega(\alpha)$ and $l_{2}(\alpha) = \Omega(\alpha)$ are defined in Lemma \ref{lem: Decrease for large gradient} and Lemma \ref{lem: Decrease for strict saddle}.

    Defining event $\mathcal{E}_{i} := \{ (\exists j\leq i)~\hat{\mathbf{x}}^{\kappa_j} \in \mathcal{L}_{\alpha,\mu,\delta}^{3}\}$, by law of total expectation,
    \begin{multline*}
        \mathbb{E}[Q_{\alpha} (\hat{\mathbf{x}}^{\kappa_{i+1}}) -Q_{\alpha} (\hat{\mathbf{x}}^{\kappa_{i}})]\\ 
        = \mathbb{E}[Q_{\alpha} (\hat{\mathbf{x}}^{\kappa_{i+1}}) -Q_{\alpha} (\hat{\mathbf{x}}^{\kappa_{i}}) ~|~ \mathcal{E}_{i}] \cdot \mathbb{P}[\mathcal{E}_{i}]\\
        + \mathbb{E}[Q_{\alpha} (\hat{\mathbf{x}}^{\kappa_{i+1}}) -Q_{\alpha} (\hat{\mathbf{x}}^{\kappa_{i}}) ~|~ \mathcal{E}_{i}^{c}] \cdot \mathbb{P}[\mathcal{E}_{i}^{c}].
    \end{multline*}
    Combining \eqref{eq: stochastic process} and \eqref{eq: expectation over L1 and L2} gives
    \begin{align*}
        \mathbb{E}[Q_{\alpha} (\hat{\mathbf{x}}^{\kappa_{i+1}}) - Q_{\alpha} (\hat{\mathbf{x}}^{\kappa_{i}}) ~|~ \mathcal{E}_{i}^{c}]
        \le -l(\alpha)\cdot \Delta\kappa_{i},   
    \end{align*}
    where $l(\alpha) = \min\{l_{1}(\alpha), l_{2}(\alpha)/T_{\max}(\alpha)\} = \Omega(\alpha^{2})$ and $\Delta\kappa_{i} = \kappa_{i+1} - \kappa_{i}$. Since $\mathbb{P}[\mathcal{E}_{i-1}] \le \mathbb{P}[\mathcal{E}_{i}]$, we obtain
    \begin{multline*}
        \mathbb{E}[Q_{\alpha} (\hat{\mathbf{x}}^{\kappa_{i+1}})] - \mathbb{E}[Q_{\alpha} (\hat{\mathbf{x}}^{\kappa_{i}})]\\
        \le \mathbb{E}[Q_{\alpha} (\hat{\mathbf{x}}^{\kappa_{i}}) ~|~ \mathcal{E}_{i}] \cdot (\mathbb{P}[\mathcal{E}_{i}] - \mathbb{P}[\mathcal{E}_{i-1}])-l(\alpha) \cdot \Delta\kappa_{i}.
    \end{multline*}
    Since the generated sequence $\{Q_{\alpha}(\hat{\mathbf{x}}^{k})\}$ is assumed bounded, there exists $b>0$ such that $\|Q_{\alpha}(\hat{\mathbf{x}}^{k})\| \le b$ for all $k = 0,1,\cdots$. As such,
    \begin{multline*}
        \mathbb{E}[Q_{\alpha} (\hat{\mathbf{x}}^{\kappa_{i+1}})] - \mathbb{E}[Q_{\alpha} (\hat{\mathbf{x}}^{\kappa_{i}})]\\
        \le b \cdot (\mathbb{P}[\mathcal{E}_{i}] - \mathbb{P}[\mathcal{E}_{i-1}])
        -l(\alpha) \cdot \Delta\kappa_{i} \cdot \mathbb{P}[\mathcal{E}_{i}^{c}].
    \end{multline*}
    Summing both sides of the inequality over $i$ gives
    \begin{multline*}
        \mathbb{E}[Q_{\alpha} (\hat{\mathbf{x}}^{\kappa_{i}})] - \mathbb{E}[Q_{\alpha} (\hat{\mathbf{x}}^{\kappa_{1}})]\\
        \le b \cdot (\mathbb{P}[\mathcal{E}_{i-1}] - \mathbb{P}[\mathcal{E}_{0}])
        -l(\alpha) \cdot (\kappa_{i} - \kappa_{1}) \cdot \mathbb{P}[\mathcal{E}_{i}^{c}].
    \end{multline*}
    Since $\|Q_{\alpha}(\hat{\mathbf{x}}^{k})\| \le b$ for all $k = 0,1,\cdots$, it follows that $-2b \le b - l(\alpha) \cdot (\kappa_{i} - \kappa_{1}) \cdot \mathbb{P}[\mathcal{E}_{i}^{c}]$, which leads to the following upper bound for the probability of event $\mathcal{E}_{i}^{c}$:
    \begin{align*}
        \mathbb{P}[\mathcal{E}_{i}^{c}] \le \frac{3b}{l(\alpha)(\kappa_{i} - \kappa_{1})}.
    \end{align*}
    Therefore, if $\kappa_{i} - \kappa_{1}$ grows larger than $6b / l(\alpha)$, then $\mathbb{P}[\mathcal{E}_{i}^{c}] \le 1 / 2$. Since $\kappa_{1} \le T_{\max}(\alpha) = \mathcal{O}(\alpha^{-1})$, after $K^{\prime}(\alpha) = 6b / l(\alpha) + T_{\max}(\alpha) = \mathcal{O}(\alpha^{-2})$ steps, $\{\hat{\mathbf{x}}^{k}\}$ must enter $\mathcal{L}_{\alpha,\mu,\delta}^{3}$ at least once with probability at least $1 / 2$. Therefore, by repeating this step $\log\zeta^{-1}$ times, the probability of $\{\hat{\mathbf{x}}^{k}\}$ entering $\mathcal{L}_{\alpha,\mu,\delta}^{3}$ at least once is lower bounded:
    \begin{align*}
        \mathbb{P}[\{(\exists k\leq K(\alpha,\zeta))~\hat{\mathbf{x}}^{k} \in \mathcal{L}_{\alpha,\mu,\delta}^{3}\}] \ge 1 - \frac{\zeta}{2},
    \end{align*}
    where $K(\alpha) = \mathcal{O}(\alpha^{-2}\log\zeta^{-1})$. Combining this with Lemma \ref{lem: Stay close to minimizers}, we have that, after $K(\alpha,\zeta)$ iterations, Algorithm \ref{alg: Noisy dgd} produces a point that is $\Delta_{2}(\alpha,\zeta)$-close to $\hat{\mathcal{X}}^{*}_{Q_{\alpha}}$ with probability at least $1-\zeta$, where $\Delta_{2}(\alpha,\zeta) = \mathcal{O}(\sqrt{\alpha \log(\alpha^{-1} \zeta^{-1})})$. For given $\Delta_{1} > 0$, since $\alpha \le \Bar{\alpha}(\Delta_{1})$ satisfies requirements of Theorem \ref{the: Local minimizers for Q},
    $\hat{\mathbf{x}}^{*}=\inf_{\hat{\mathbf{x}} \in \mathcal{X}_{Q_{\alpha}}^{*}}\|\hat{\mathbf{x}}_{i}^{K(\alpha,\zeta)} -\hat{\mathbf{x}}\|$ is such that $\hat{\mathbf{x}}_{i}^{*}$ is $\Delta_{1}$-close to $\mathcal{X}_{f}^{*}$. To summarize, we have for $i \in \mathcal{V}$,
    \begin{align*}
        \mathrm{dist}(\hat{\mathbf{x}}_{i}^{K(\alpha,\zeta)},\mathcal{X}_{f}^{*}) \leq \Delta_{1} + \Delta_{2}(\alpha,\zeta)
    \end{align*}
    as claimed.
\end{proof}

\section{Numerical Example}
\label{sec: Numerical Example}
Consider the following non-convex optimization problem over $\mathbf{x} = (x_1,x_2)$:
\begin{align*}
    \min_{\mathbf{x} \in \mathbb{R}^2} f({\mathbf{x}}) = \min_{\mathbf{x} \in \mathbb{R}^2} \sum_{i=1}^5 f_i(\mathbf{x}) = x_1^4-x_1^2+x_2^4+x_2^2,
\end{align*}
where $f_1(\mathbf{x}) = 0.25 x_1^4 - x_1^2 - x_2^2$, $f_2(\mathbf{x}) = 0.25 x_1^4 + 0.5 x_2^4 + 1.5 x_2^2$, $f_3(\mathbf{x}) = -x_1^2 + x_2^2$, $f_4(\mathbf{x}) = 0.5 x_1^4 - 0.5 x_2^2$, and $f_5(\mathbf{x}) = x_1^2 + 0.5 x_2^4$.
The mixing matrix is taken to be
\begin{align*}
    \mathbf{W} = 
    \begin{bmatrix} 
        0.6 & 0 & 0.2 & 0 & 0.2 \\
        0 & 0.6 & 0 & 0.2 & 0.2 \\
        0.2 & 0 & 0.6 & 0.2 & 0 \\
        0 & 0.2 & 0.2 & 0.6 & 0 \\
        0.2 & 0.2 & 0 & 0 & 0.6
    \end{bmatrix}.
\end{align*}
It can be verified that $\mathbf{x} = (x_1, x_2) = (0, 0)$ is a saddle point of $f$, and that $\mathbf{x} = (x_1, x_2) = (-\frac{\sqrt{2}}{2}, 0)$ and $\mathbf{x} = (x_1, x_2) = (\frac{\sqrt{2}}{2}, 0)$ are two local minimizers. We compare the performance of \textbf{DGD} and \textbf{NDGD} with constant step-size $\alpha = 0.005$, both initialized from $\mathbf{x}^0 = (x_1^0, x_2^0) = (10^{-6}, 10^{-6})$ (i.e., close to a saddle point). 



    \begin{figure*}
        \centering
        \begin{subfigure}[h]{0.42\textwidth}
            \centering
            \includegraphics[width=\textwidth]{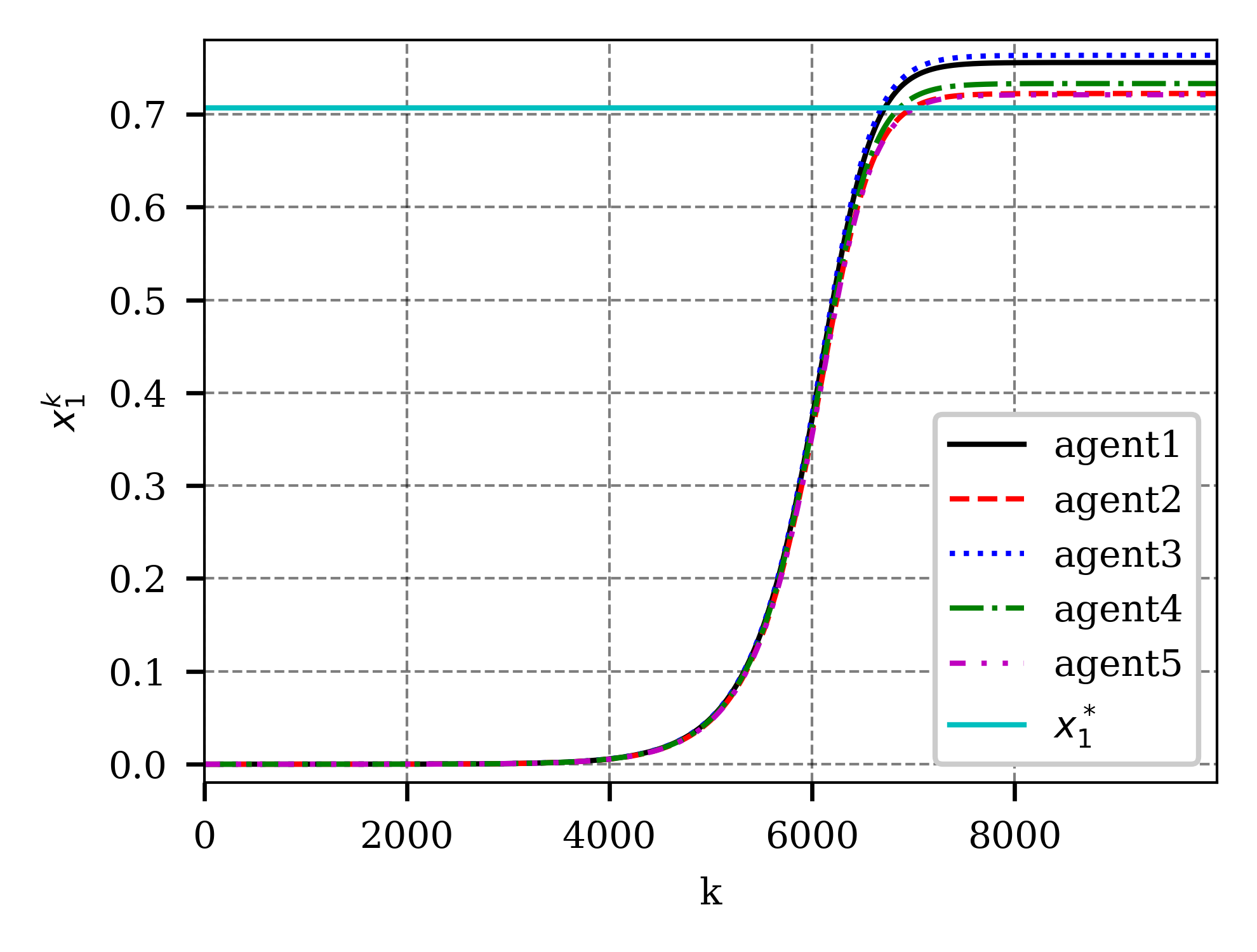}
            \caption{$\mathbf{x}_1^k$ of DGD.}
            \label{fig: DGD}
        \end{subfigure}\hspace{0.15\textwidth}
        \begin{subfigure}[h]{0.42\textwidth}
            \centering
            \includegraphics[width=\textwidth]{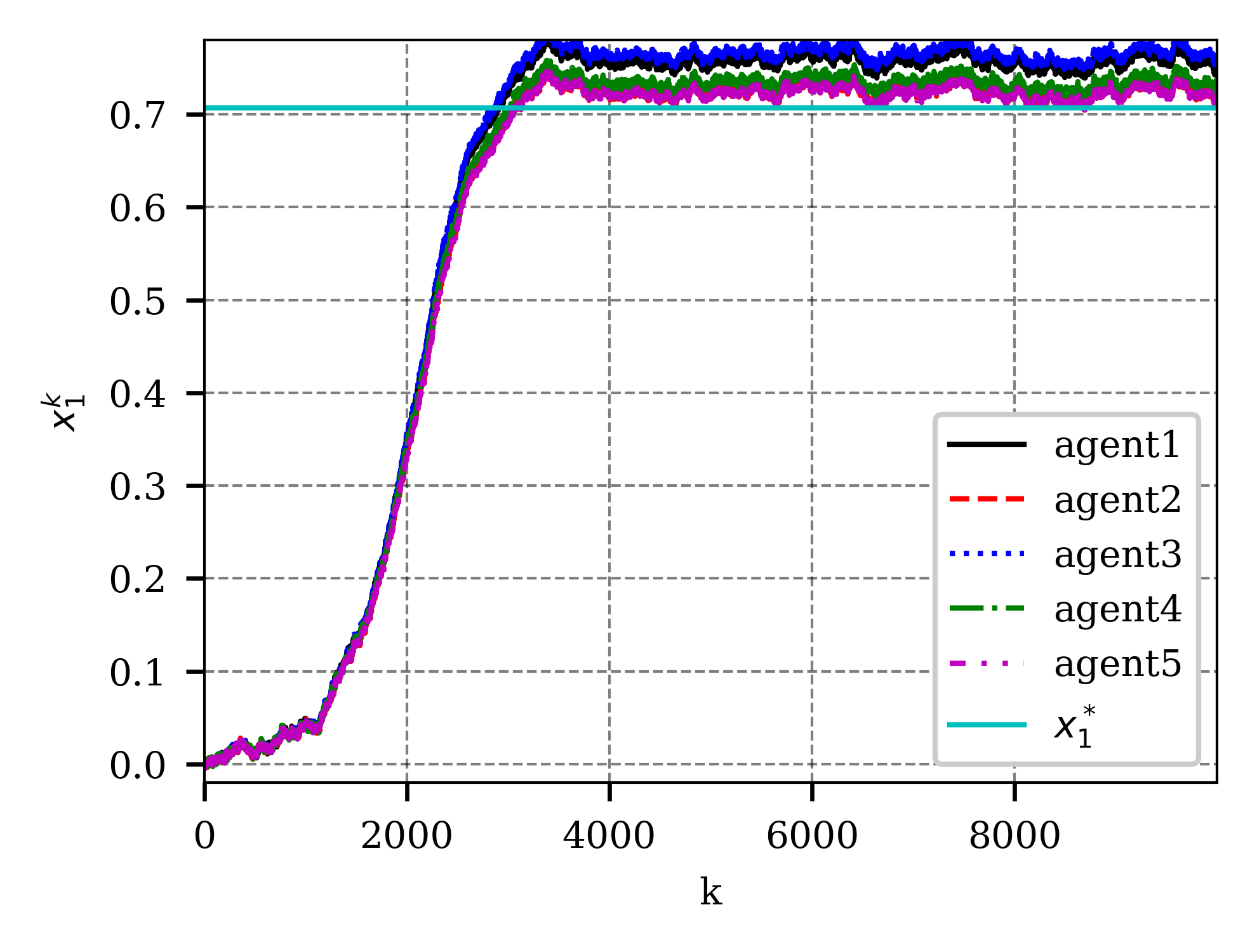}
            \caption{$\mathbf{x}_1^k$ of NDGD.}
            \label{fig: NDGD}
        \end{subfigure}
        \centering
        \caption{Iterations of $\mathbf{x}_1$.}
        \label{fig: x1}
    \end{figure*}
    
From Figure \ref{fig: x1}(\subref{fig: DGD}), although not trapped forever, it does take \textbf{DGD} about 6000 iterations to escape the vicinity of the saddle point and converge to the neighborhood of a local minimizer. From Figure \ref{fig: x1}(\subref{fig: NDGD}), we can see that \textbf{DGD} escapes the vicinity of the saddle point with about 2000 iterations and converges to the neighborhood of a local minimizer. The effectiveness of \textbf{NDGD} over \textbf{DGD} is evident through this example.

\section{Conclusion}
\label{sec: Conclusion}
A fixed step-size noisy distributed gradient descent (\textbf{NDGD}) algorithm is formulated for solving optimization problems in which the objective is a finite sum of smooth but possibly non-convex functions. Random perturbations are added to the gradient descent at each step to actively evade saddle points. Under certain regularity conditions, and with a suitable step-size, each agent converges (in probability with specified confidence) to a neighborhood of a local minimizer. In particular, we determine a probabilistic upper bound on the distance between the iterate at each agent, and the set of local minimizers, after a sufficient number of iterations.

The potential applications of the \textbf{NDGD} algorithm are vast and varied, including multi-agent systems control, federated learning and sensor networks location estimation, particularly in large-scale network scenarios. Further exploration of different approaches to introducing random perturbations, and analysis of convergence rate performance can be pursued in future work.

\bibliographystyle{IEEEtran}
\bibliography{IEEEabrv,Reference}

\end{document}